\documentclass[11pt]{amsart}

\usepackage{amssymb}
\usepackage{amsmath}
\usepackage{amsfonts}
\usepackage{graphicx}
\usepackage{subfigure}
\usepackage{amsthm}
\usepackage{enumerate}
\usepackage[mathscr]{eucal}
\usepackage{mathrsfs}
\usepackage{verbatim}
\usepackage{yhmath}
\usepackage{epstopdf}
\usepackage{color}
\usepackage[colorlinks,linkcolor=red,anchorcolor=green,citecolor=blue]{hyperref}

\makeatletter
\@namedef{subjclassname@2020}{%
  \textup{2020} Mathematics Subject Classification}
\makeatother

\numberwithin{equation}{section}
\numberwithin{figure}{section}

\theoremstyle{plain}
\newtheorem{theorem}{Theorem}[section]
\newtheorem{lemma}[theorem]{Lemma}
\newtheorem{proposition}[theorem]{Proposition}
\newtheorem{corollary}[theorem]{Corollary}

\newtheorem{ques}[theorem]{Question}

\theoremstyle{plain}

\theoremstyle{remark}
\newtheorem{remark}[theorem]{Remark}

\allowdisplaybreaks[4]

\begin{document}
\date{}

\title[P\'{o}lya's conjecture]{Improvement of P\'{o}lya's conjecture for balls and cylinders}

\author{Jingwei Guo}
\address{School of Mathematical Sciences\\
University of Science and Technology of China\\
Hefei, 230026\\ P.R. China}
\email{jwguo@ustc.edu.cn}

\author{Changxing Miao}
\address{Institute of Applied Physics \& Computational Mathematics\\
Beijing, 100088\\P.R. China}
\email{miao\_changxing@iapcm.ac.cn}

\author{Weiwei Wang}
\address{School of Mathematical Sciences\\
University of Electronic Science and Technology of China\\
Chengdu, 611731\\ P.R. China}
\email{wawnwg123@163.com}

\author{Guoqing Zhan}
\address{School of Mathematical Sciences\\
University of Science and Technology of China\\
Hefei, 230026\\ P.R. China}
\email{zhanguoqing@mail.ustc.edu.cn}

\thanks{}



\begin{abstract}
P\'{o}lya's conjecture on the eigenvalues of the Laplacian has been one of the core problems in spectral geometry. Building upon the recent breakthrough works on P\'{o}lya's conjecture for balls and annuli by Filonov, Levitin, Polterovich and Sher, we study several aspects of P\'{o}lya's conjecture for balls and cylinders: by refining the purely analytical portion of the proof in \cite{FLPS:2023} for the Neumann P\'{o}lya's conjecture for the disk, we extend the regime of the spectral parameter that can be established without computer assistance; we obtain improvement of P\'{o}lya's conjecture  for disks and balls;  we obtain improvement of P\'{o}lya's conjecture for cylinders and  confirm the Neumann P\'{o}lya's conjecture for cylinders in $\mathbb{R}^3$. As a supplementary effort, we study Weyl's law for cylinders.

\end{abstract}

\maketitle

\tableofcontents

\section{Introduction} \label{intro}

Consider the Laplacian on a bounded domain $\mathscr{D}\subset \mathbb{R}^d$ ($d\geq 2$), under either Dirichlet or Neumann boundary conditions. Denote by
\begin{equation*}
    \mathscr{N}_\mathscr{D}(\mu)=\#\{j\ | \lambda_j \le \mu\}
\end{equation*}
the corresponding eigenvalue counting function, where $\lambda_j^2$ are the eigenvalues for the chosen (Dirichlet or Neumann) boundary condition. It has an asymptotics in the following form
\begin{equation}\label{s1-7}
	\mathscr{N}_\mathscr{D}(\mu)=C_d \left|\mathscr{D}\right|\mu^d\mp C'_d\left|\partial\mathscr{D}\right| \mu^{d-1}+\mathscr{R}_\mathscr{D}(\mu) \quad \textrm{as $\mu\rightarrow \infty$,}
\end{equation}
where $C_d=\omega_d/(2\pi)^{d}$, $C'_d=\omega_{d-1}/(4(2\pi)^{d-1})$, and $\omega_k$ is the volume of the unit ball in $\mathbb{R}^k$. Here, the minus sign (respectively, plus sign) corresponds to the Dirichlet (respectively, Neumann) boundary condition, and $\mathscr{R}_\mathscr{D}(\mu)$ is the Weyl remainder. To distinguish between the Dirichlet and Neumann cases, we may write $\mathscr{N}^{\mathtt{D}}_\mathscr{D}(\mu)$ or $\mathscr{N}^{\mathtt{N}}_\mathscr{D}(\mu)$ with superscripts; otherwise, we omit the superscripts when there is no confusion.

The study of the eigenvalue counting function lies at the heart of spectral geometry, boasting a rich history and a wealth of profound results. This field is primarily concerned with understanding the deep connection between the spectrum of the Laplacian and the underlying geometry of a manifold. Among its most fundamental achievements and enduring puzzles are Weyl's law, which establishes the asymptotic growth of the counting function, and P\'{o}lya's conjecture, which proposes a sharp geometric bound for domains in Euclidean space.

In this paper, we address specific questions surrounding Weyl's law and  P\'{o}lya's conjecture for balls and cylinders. Denote by
\begin{equation*}
	\mathcal{B}_R^d\subset \mathbb{R}^d
\end{equation*}
the ball centered at the origin with radius $R>0$  and by
\begin{equation*}
	\mathcal{C}_{R,L}^d=\mathcal{B}_R^{d-1}\times [0,L]\subset \mathbb{R}^d
\end{equation*}
the cylinder with radius $R>0$ and height $L>0$. We provide a sharper upper bound for the remainder term in Weyl's law for cylinders.

\begin{theorem}\label{thm111}
For integer $d\geq 3$, if there exists some constant $\alpha\in (0,1)$ such that, for any $\varepsilon>0$,
\begin{equation*}
    \mathscr{R}_{\mathcal{B}_R^{d-1}}(\mu)=O_{\varepsilon}\left(\mu^{d-3+\alpha+\varepsilon} \right),
\end{equation*}
then, for any $\varepsilon>0$,
\begin{equation*}
    \mathscr{R}_{\mathcal{C}_{R,L}^d}(\mu)=O_{\varepsilon}\left(\mu^{d-2+\alpha+\varepsilon} \right).
\end{equation*}
This result holds for both the Dirichlet and Neumann cases.
\end{theorem}

Based on recent results on Weyl's law for balls in \cite[Theorem 1.1]{GJWY:2024}, $\alpha$ can take the value
\begin{equation*}
    2\theta^*=0.628966\cdots
\end{equation*}
with $-\theta^*$ defined as a certain solution to an elementary equation (see \cite[Theorem 6.1]{GJWY:2024} for its precise definition). Hence, we have

\begin{corollary} \label{cor1}
For integer $d\geq 3$ and $\varepsilon>0$,
\begin{equation*}
    \mathscr{R}_{\mathcal{C}_{R,L}^d}(\mu)=O_{\varepsilon}\left(\mu^{d-2+2\theta^*+\varepsilon} \right).
\end{equation*}
\end{corollary}
\noindent  This improves upon the remainder estimate $o(\mu^{d-1})$ in the Weyl conjecture for general domains and manifolds with boundary (see, for example, Ivrii \cite{I:1980}).       Thus, we now have a better understanding of the behavior of $\mathscr{N}_{\mathcal{C}_{R,L}^d}(\mu)$ as $\mu \to \infty$.

Unlike the study of Weyl's law, P\'{o}lya \cite{Polya:1954} conjectured in 1954 that the eigenvalue counting function is bounded by the leading term in the asymptotics \eqref{s1-7} from below (above) for the Neumann (Dirichlet) boundary condition. That is,  for all $\mu\geq 0$,
\begin{equation}\label{PolyaConj}
    \mathscr{N}^{\mathtt{D}}_{\mathscr{D}}(\mu)\leq C_d|\mathscr{D}|\mu^d \leq \mathscr{N}^{\mathtt{N}}_{\mathscr{D}}(\mu).
\end{equation}
The first and second inequalities in \eqref{PolyaConj} are referred to as the Dirichlet and Neumann P\'{o}lya's conjectures, respectively.

In 1961, P\'{o}lya \cite{Polya:1961} proved his conjecture \eqref{PolyaConj} for planar tiling domains. While his proof for the Neumann case relied on an additional ``regular tiling'' assumption, this restriction was removed by Kellner \cite{Kellner:1966} in 1966. In 1984, Urakawa \cite{Urakawa:1983} generalized P\'{o}lya's result to bounded domains in $\mathbb{R}^d$ for the Dirichlet case. This derivation, which incorporated an extra factor of the lattice packing density, confirmed the conjecture for higher-dimensional tiling domains. In 1997, Laptev \cite{Laptev:1997} showed that the Dirichlet P\'{o}lya's conjecture holds for the product $\Omega_1\times\Omega_2$ whenever it holds for $\Omega_1\subset \mathbb{R}^{d_1}$ with $d_1\geq 2$, and $\Omega_2\subset \mathbb{R}^{d_2}$ is of finite Lebesgue measure. In 2023, Freitas and Salavessa \cite{FS:2023} explored Laptev's approach and showed the existence of classes of non-tiling domains satisfying P\'{o}lya’s conjecture in any dimension, in both the Euclidean and non-Euclidean cases (see also Freitas, Lagac\'e, and Payette \cite{FLP:2021}). Recently, He and Wang \cite{HW:2024} proved P\'{o}lya's conjecture for certain thin products of Euclidean domains of the form $(a\Omega_1)\times \Omega_2$ for small $a>0$. We refer interested readers to \cite{FLP:2021,HW:2024} and the references therein for further results on unions and products of domains. Recent progress has also been made on variants of the conjecture. Notably, Filonov \cite{F:2025} tackled the Neumann case for planar convex domains, and Jiang and Lin \cite{JL:2025} proved an $\epsilon$-loss version of the Dirichlet case for bounded Lipschitz domains in $\mathbb{R}^d$ based on a uniform estimate of the remainder of Weyl's law, and provided in all dimensions a class of general domains satisfying the Dirichlet P\'{o}lya's conjecture.

In a remarkable 2023 paper, Filonov, Levitin, Polterovich, and Sher \cite{FLPS:2023} achieved a breakthrough by proving P\'{o}lya's  conjecture for the disk. This work resolved a long-standing open problem, establishing the disk as the first non-tiling  planar domain for which the celebrated conjecture is known to hold. Their work also encompassed the confirmation of the conjecture for arbitrary planar sectors and, in the Dirichlet case, for balls of any dimension, showcasing the power and versatility of their methodology. Their very recent follow-up work \cite{FLPS:2025} has further expanded the frontier by verifying the Dirichlet P\'{o}lya's conjecture for annuli, another fundamental and geometrically non-trivial domain, thus providing the first verification of the conjecture for a non-simply connected planar domain.

In their proof of the Neumann P\'{o}lya's conjecture for the disk, Filonov, Levitin, Polterovich, and Sher \cite{FLPS:2023} provided a purely analytic proof of the following theorem.

\begin{theorem}[{\cite[Theorem 1.4]{FLPS:2023}}]  \label{thm222}
The Neumann P\'{o}lya's conjecture for the unit disk holds for all $\mu\geq
\frac{6\pi}{3\pi-8}$.
\end{theorem}

\noindent Note that $6\pi/(3\pi-8)\approx 13.23$. The regime $\mu\leq 2\sqrt{3}\approx 3.46$ is easy to verify analytically (see \cite[Lemma 1.3]{FLPS:2023}), while the intermediate interval $\mu \in [3, 14]$ was settled in \cite[Theorem 1.5]{FLPS:2023} by a rigorous computer-assisted argument. Consequently, the remaining challenge for a complete analytical understanding is the spectral gap $(2\sqrt{3}, 6\pi/(3\pi-8))$.

A main result of this paper is the following theorem, which we prove analytically in Subsection \ref{sec2.1}.

\begin{theorem} \label{thm999}
The Neumann P\'{o}lya's conjecture for the unit disk holds for all $\mu \geq 11.89$.
\end{theorem}
\noindent It improves upon the previous regime of $\mu \geq 6\pi/(3\pi-8)$. Thus, the challenge narrows to providing an analytical proof for the interval $(2\sqrt{3}, 11.89)$. This result is achieved by refining the analytical method of Filonov, Levitin, Polterovich, and Sher. Our improvement comes from exploiting more comprehensive information about the function $g$ in \eqref{s2-9}---a fundamental element in the eigenvalue analysis of the Laplacian on disks and balls. Specifically, we use not only its first and second derivatives, but also its third derivative.  The function $g$ emerges naturally in the asymptotics of Bessel functions. Furthermore, it has a notable interpretation in probability theory, specifically concerning the distribution of the maximum absolute value of a Brownian bridge.


We now turn to a different perspective on P\'{o}lya's conjecture. Note that the remainder term in the asymptotics \eqref{s1-7} is of lower order than the second main term (if exists). See, for example, Corollary \ref{cor1}. Therefore, for sufficiently large $\mu$, the second main term dominates the remainder, allowing us to obtain a bound of $\mathscr{N}_\mathscr{D}(\mu)$ sharper than the one predicted by P\'{o}lya's conjecture. Motivated by this simple observation, a natural question arises: do such improved bounds hold for all $\mu\geq 0$? That is, do we have
\begin{equation}
    \mathscr{N}^{\mathtt{D}}_{\mathscr{D}}(\mu)\leq C_d|\mathscr{D}|\mu^d-C \mu^{d-1} \label{s1-1}
\end{equation}
and
\begin{equation}
    \mathscr{N}^{\mathtt{N}}_{\mathscr{D}}(\mu)\geq C_d|\mathscr{D}|\mu^d+C' \mu^{d-1}  \label{s1-2}
\end{equation}
for some positive constants $C$ and $C'$ and all $\mu\geq 0$? Unfortunately, the proposed inequality \eqref{s1-1} fails to hold in general. This bound is contradicted by the following elementary observation: if $\mu^2$ is smaller than the first Dirichlet eigenvalue, then $\mathscr{N}^{\mathtt{D}}_{\mathscr{D}}(\mu)$ vanishes. Consequently, for any fixed $C>0$, the inequality fails for sufficiently small $\mu$ (even though P\'{o}lya's original conjecture may still hold).  In contrast, \eqref{s1-2} holds trivially for small $\mu$ due to the fact that $\mathscr{N}^{\mathtt{N}}_{\mathscr{D}}(\mu) \geq 1$; our interest therefore lies in its validity in the non-trivial regime.

Instead, one may formulate the following question regarding an improvement of P\'{o}lya's conjecture.
\begin{ques}
Do inequalities \eqref{s1-1} and \eqref{s1-2} hold for some positive constants $C$ and $C'$, and for all $\mu$ greater than some fixed positive constant?
\end{ques}
\noindent In another word, we are curious about the ``in-between'' case (that is, between ``large $\mu$'' and ``nonnegative $\mu$''). We seek to characterize the behavior for values of $\mu$ that are positive yet not necessarily large, thereby addressing the gap between these two previously studied cases.


Very recently, Filonov, Levitin, Polterovich and Sher's \cite[Theorem 1.5]{FLPS:2025} provided such an improvement for the Dirichlet Laplacian on the unit disk:
\begin{equation}
    \mathscr{N}^{\mathtt{D}}_{\mathcal{B}_1^2}(\mu)<\frac{\mu^2}{4}-\frac{\lfloor\omega_0 \mu \rfloor}{2} \label{s1-4}
\end{equation}
with
\begin{equation}\label{s1-5}
\omega_0=\frac{\sqrt{3}}{2\pi}-\frac{1}{6}\approx 0.108998.
\end{equation}
For $\mu\geq 1/\omega_0$, this result establishes a bound that is strictly sharper than that of P\'{o}lya's conjecture. This improvement is based on a refined lattice-point counting estimate, which is a pivot part of  their proof of the  Dirichlet P\'{o}lya's conjecture for planar annuli.

Later, Jiang and Lin's \cite[Theorem 2.3]{JL:2025} obtained an improvement for the Dirichlet Laplacian on the product domain $\Omega_1\times\Omega_2\subset \mathbb{R}^{d_1}\times\mathbb{R}$, with $d_1\geq 1$ and $\Omega_1$, $\Omega_2$ bounded, under the assumption that the Dirichlet P\'{o}lya's conjecture holds for $\Omega_1$.

Another main result of this paper is a series of improvement of P\'{o}lya's conjecture
for disks, balls, and cylinders.

\begin{theorem} \label{thm333}
The following improvement of P\'{o}lya's conjecture holds for the Neumann Laplacian on the unit disk
\begin{equation*}
\mathscr{N}^{\mathtt{N}}_{\mathcal{B}_1^2}(\mu)\geq \mu^2/4+0.0014\mu
\end{equation*}
for all $\mu\geq 12$.
\end{theorem}

We extend \cite[Theorem 1.5]{FLPS:2025} in two directions: improved estimates in two dimensions and a generalization to higher dimensions.

\begin{theorem} \label{thm777}
The following improvement of P\'{o}lya's conjecture holds for the Dirichlet Laplacian on the unit disk
\begin{equation*}
    \mathscr{N}^{\mathtt{D}}_{\mathcal{B}_1^2}(\mu)\leq \frac{\mu^2}{4}-\frac{1}{8}\lfloor\omega_0 \mu \rfloor-\frac{3}{8}\lfloor\omega_1 \mu \rfloor
\end{equation*}
for all $\mu\geq 0$, with $\omega_0$ defined by \eqref{s1-5} and
\begin{equation}\label{s1-6}
    \omega_1=\frac{\sqrt{2+\sqrt{2}}}{2\pi}-\frac{3\sqrt{2-\sqrt{2}}}{16}\approx 0.150574.
\end{equation}
\end{theorem}
For $\mu\geq 1/\omega_1$, the above bound is sharper than that of P\'{o}lya's conjecture.

\begin{theorem} \label{thm888}
The following improvement of P\'{o}lya's conjecture holds for the Dirichlet Laplacian on the unit ball in $\mathbb{R}^d$ ($d\geq 3$)
\begin{align*}
    \mathscr{N}^{\mathtt{D}}_{\mathcal{B}_1^d}(\mu)\leq &\frac{\omega_{d}^2}{(2\pi)^{d}} \mu^d-\frac{1}{2}\lfloor\omega_0 \mu \rfloor \binom{\lfloor \frac{\mu}{2}-\frac{d}{2}+1\rfloor+d-2}{d-2}\\
    &-\frac{3}{8}\left(\lfloor\omega_1 \mu \rfloor-\lfloor\omega_0 \mu \rfloor \right)\binom{\lfloor \cos(\frac{3}{8}\pi)\mu-\frac{d}{2}+1\rfloor+d-2}{d-2}
\end{align*}
for all $\mu\geq 0$, with $\omega_0$ and $\omega_1$ defined by \eqref{s1-5} and \eqref{s1-6} respectively. Here, we follow the convention that $\binom{n}{k} = 0$ for $k > n$.
\end{theorem}

We establish the following improved form of P\'{o}lya's conjecture for cylinders. 

\begin{theorem} \label{thm444}
For all $\mu\geq 0$, the following improvement of P\'{o}lya's conjecture holds for the Dirichlet Laplacian on cylinders in $\mathbb{R}^3$
\begin{equation*}
    \mathscr{N}^{\mathtt{D}}_{\mathcal{C}_{R,L}^3}(\mu)\leq \frac{1}{6\pi} R^2 L\mu^3-\frac{\pi^2 R^2}{8L^2} \left\lfloor\frac{L\mu}{\pi} \right\rfloor\left(\left\lfloor\frac{L\mu}{\pi} \right\rfloor+\frac{1}{3} \right)
\end{equation*}
and on cylinders in $\mathbb{R}^d$ ($d\geq 4$)
\begin{equation}
\begin{split}
    \mathscr{N}^{\mathtt{D}}_{\mathcal{C}_{R,L}^d}(\mu)<&\frac{\omega_{d}}{(2\pi)^{d}}\left|\mathcal{C}_{R,L}^d\right| \mu^d-\frac{(d-1)\omega_{d-1}}{4(2\pi)^{d-1}L} \left|\mathcal{C}_{R,L}^d\right|\cdot\\
    &\left(\frac{\pi}{L\mu}\left\lfloor\frac{L\mu}{\pi \sqrt{d-2}} \right\rfloor\right)^2 \left(1- A_1\left(\frac{\pi}{L\mu}\left\lfloor\frac{L\mu}{\pi \sqrt{d-2}} \right\rfloor\right)^2\right)\mu^{d-1}
\end{split} \label{s1-3}
\end{equation}
with
\begin{equation*}
    A_1=\left\{
            \begin{array}{ll}
           \frac{d-3}{4},  & \textrm{if $d\geq 5$,}\\
           \frac{2}{3},            &\textrm{if $d=4$.}
            \end{array}
        \right.
\end{equation*}
Furthermore, if  $\mu\geq \pi \sqrt{d-2}/L$ and $d\geq 11$, the inequality \eqref{s1-3} still holds when the term below is subtracted from its right side
\begin{equation*}
    \frac{\omega_{d-1}}{2(2\pi)^{d-1}L}\left|\mathcal{C}_{R,L}^d\right|\left(1-\left(\frac{\pi}{L\mu}\left(\left\lfloor\frac{L\mu}{\pi \sqrt{d-2}} \right\rfloor+2\right) \right)^2\right)^{\frac{d-1}{2}} \mu^{d-1}.
\end{equation*}

The following improvement of P\'{o}lya's conjecture holds for the Neumann Laplacian on cylinders in $\mathbb{R}^3$
\begin{equation*}
\mathscr{N}^{\mathtt{N}}_{\mathcal{C}_{R,L}^3}(\mu) \geq \frac{1}{6\pi} R^2 L\mu^3+\left\{
            \begin{array}{ll}
           \frac{\pi R^2}{8L} \left( \frac{L\mu}{\pi}-\frac{1}{3}\right)\mu,                                         & \textrm{if $\frac{L\mu}{\pi} \geq \frac{1}{3}$,}\\
           \frac{R^2}{6} \left( \frac{3}{2}-\frac{L\mu}{\pi}\right)\mu^2,                                   &\textrm{if $0\leq \frac{L\mu}{\pi} \leq \frac{3}{2}$.}
            \end{array}
        \right.
\end{equation*}
\end{theorem}

\begin{remark}
    In Section \ref{s4.2}  we prove an improvement of P\'{o}lya's conjecture for product domains slightly more general than cylinders, though we do not attempt to extend the result to a fully general setting.  Theorem \ref{thm444} follows directly from Theorems \ref{thm555} and \ref{thm666} below combined with \cite[Theorem 1.2 and Corollary 1.6]{FLPS:2023}.

While our Theorem \ref{thm555} addresses the Dirichlet case in a manner analogous to Jiang and Lin's Theorem 2.3, it relies only on elementary analysis and leads to distinct bounds. The corresponding Neumann case is handled in our Theorem \ref{thm666}.

    From the proof of Theorem \ref{thm666}, it is easy to observe that: the Neumann P\'{o}lya's conjecture holds for the product $\mathscr{D}\times I$ whenever it holds for $\mathscr{D}\subset \mathbb{R}^{d-1}$, with $d\geq 3$, and $I$ is a finite interval.
\end{remark}

Theorem \ref{thm444} directly confirms the Neumann P\'{o}lya's conjecture for cylinders in $\mathbb{R}^3$ (while the Dirichlet P\'{o}lya's conjecture is already known to hold for cylinders in $\mathbb{R}^d$, $d\geq 3$, as a consequence of \cite[Theorem 1.2]{FLPS:2023} and \cite[Theorem 2.8]{Laptev:1997}).

\begin{corollary}
    The Neumann P\'{o}lya's conjecture  holds for the cylinder $\mathcal{C}_{R,L}^3$.
\end{corollary}

The case of sufficiently small $L$ has been established by He and Wang \cite[Theorem 1.2]{HW:2024} (see also Freitas and Salavessa \cite{FS:2023}). Our result, in contrast, applies for all $L>0$.

\begin{remark}
At the end of the Introduction, we would like to pose a potentially interesting question: what are the optimal values of the constants $C$ and $C'$ in \eqref{s1-1} and \eqref{s1-2}, depending on the regime of the spectral parameter?
\end{remark}



\section{Disks and balls} \label{s2}

\subsection{The Neumann P\'{o}lya's conjecture for disks}  \label{sec2.1}

The goal of this subsection is to prove Theorem \ref{thm999}. Throughout Section \ref{s2}, we let
\begin{equation}\label{s2-9}
g(x)=\frac{1}{\pi}\left(\sqrt{1-x^2}-x\arccos x\right), \quad 0\leq x\leq 1,
\end{equation}
and
\begin{equation*}
G_\mu(x)=\mu g\left(\frac{x}{\mu}\right).
\end{equation*}
These two functions are both strictly decreasing functions of $x$. They are fundamental in the analysis of eigenvalues of the Laplacian on disks and balls. The following lemma is a key step in our analysis, as it incorporates information not only from the first and second derivatives of $g$, but also from the third derivative.

\begin{lemma} \label{s2-1}
For $\mu>0$, $0<a<\mu/2$ and $0\leq x\leq \mu-2a$, the function $G_\mu$ satisfies
\begin{equation*}
    \frac{G_\mu(x+a)-G_\mu(x+2a)}{G_\mu(x)-G_\mu(x+a)}<1-\frac{a}{\mu}\frac{g''(x/\mu)}{|g'(x/\mu)|}
\end{equation*}
and, in particular,
\begin{equation*}
    \frac{G_\mu(x+a)-G_\mu(x+2a)}{G_\mu(x)-G_\mu(x+a)}<1-\frac{2a}{\pi\mu}.
\end{equation*}
\end{lemma}

\begin{proof}
Applying the Cauchy mean value theorem to the ratio on the left side gives
\begin{equation*}
    R(x):=\frac{G_\mu(x+a)-G_\mu(x+2a)}{G_\mu(x)-G_\mu(x+a)}=\frac{G_\mu'(\theta+a)}{G_\mu'(\theta)}
\end{equation*}
for some $\theta\in (x, x+a)$. It is easy to check for $t\in (0,1)$ that
\begin{equation*}
    g'(t)<0 \textrm{ and } g'''(t)>0.
\end{equation*}
By using the Taylor expansion of $G_\mu'(\theta+a)$ at the point $\theta$ retaining two terms plus a remainder term, we obtain
\begin{equation*}
    R(x)<1+a\frac{G_\mu''(\theta)}{G_\mu'(\theta)}= 1-\frac{a}{\mu}\frac{g''(x/\mu)}{|g'(x/\mu)|}.
\end{equation*}
Notice that $g''$ is strictly increasing and $|g'|$ is strictly decreasing. Hence, the function $-g''(x)/|g'(x)|$ attains its maximum at $x = 0$.
\end{proof}

The above lemma thus yields a bound sharper than the one following from the strict convexity of $g$, which states that
   \begin{equation*}
    G_{\mu}\left(\frac{i+j}{2} \right)<\frac{1}{2}G_{\mu}(i)+\frac{1}{2}G_{\mu}(j).
\end{equation*}

\begin{lemma} \label{s2-2}
Let $0\leq i<j\leq \mu$. Then
\begin{equation*}
    G_{\mu}\left(\frac{i+j}{2} \right)<\frac{1-\Theta}{2-\Theta}G_{\mu}(i)+\frac{1}{2-\Theta}G_{\mu}(j)
\end{equation*}
with
\begin{equation}
    \Theta=\Theta(i, j, \mu)=\frac{j-i}{2\mu}\frac{g''(i/\mu)}{|g'(i/\mu)|}. \label{s2-3}
\end{equation}
\end{lemma}

\begin{proof}
For simplicity, let us define
\begin{equation*}
 l_1=G_{\mu}(i)-G_{\mu}\left(\frac{i+j}{2}\right),  \ l_2=G_{\mu}\left(\frac{i+j}{2}\right)-G_{\mu}(j) \textrm{ and } l_3=G_{\mu}(j).
\end{equation*}
Applying Lemma \ref{s2-1} with $x=i$ and $a=(j-i)/2$ yields
\begin{equation*}
    \frac{l_2}{l_1}<1-\Theta,
\end{equation*}
which implies
\begin{equation*}
    l_1+l_2>\frac{2-\Theta}{1-\Theta}l_2.
\end{equation*}
Consequently, we obtain
\begin{equation*}
    G_{\mu}\!\left(\!\frac{i+j}{2}\!\right)=l_2+l_3<\!\frac{1-\Theta}{2-\Theta}(l_1+l_2)+l_3=\frac{1-\Theta}{2-\Theta}(l_1+l_2+l_3)+\frac{1}{2-\Theta}l_3,
\end{equation*}
thus establishing the desired bound.
\end{proof}

As defined in \cite[Definition 1.3]{FLPS:2025}, for $a,b\in\mathbb{Z}$ with $a<b$, and $f:[a,b]\to\mathbb{R}$, the trapezoidal floor sum of $f$ on $[a,b]$ is given by
\begin{equation*}
	\mathbf{T}(f,a,b) = \frac{1}{2}\lfloor f(a)\rfloor+\sum_{m=a+1}^{b-1}\lfloor f(m)\rfloor + \frac{1}{2}\lfloor f(b)\rfloor.
\end{equation*}

We obtain an improved version of \cite[Lemma 6.4]{FLPS:2023}.

\begin{lemma} \label{s2-6}
For $i,j\in\mathbb{N}$ with $0\leq i<j\leq \mu$, if
\begin{equation*}
    n+\frac{1}{4}>G_{\mu}(i)\geq \cdots \geq G_{\mu}(j-1)\geq n-\frac{3}{4}>G_{\mu}(j)\geq n-\frac{7}{4}
\end{equation*}
for some $n\in\mathbb{Z}$, then
\begin{align*}
	\mathbf{T}\left(G_{\mu}+\frac{3}{4},i,j\right)\geq \int_i^j \! G_{\mu}(t) \,\textrm{d}t+\frac{j-i}{2(2-\Theta)}-\frac{1}{2}
\end{align*}
with $\Theta=\Theta(i, j, \mu)$ defined by \eqref{s2-3}.
\end{lemma}


\begin{proof}
Let $k=(i+j)/2$. Then
\begin{align*}
    \int_i^j \! G_{\mu}(t) \,\textrm{d}t&=\int_i^k \! G_{\mu}(t) \,\textrm{d}t+\int_k^j \! G_{\mu}(t) \,\textrm{d}t\\
    &\leq \frac{j-i}{4}\left(G_{\mu}(i)+2G_{\mu}(k)+G_{\mu}(j) \right),
\end{align*}
where we have bounded each integral by the area of a corresponding trapezoid.  By Lemma \ref{s2-2} and the assumption, we obtain
\begin{align*}
 \int_i^j \! G_{\mu}(t) \,\textrm{d}t&\leq   \frac{j-i}{4}\left(\frac{4-3\Theta}{2-\Theta}G_{\mu}(i)+\frac{4-\Theta}{2-\Theta}G_{\mu}(j) \right)\\
 &\leq (j-i)n-\frac{j-i}{2(2-\Theta)}.
\end{align*}
The desired inequality follows, since its left-hand side is exactly equal to $(j-i)n-\frac{1}{2}$.
\end{proof}

The following lemma is already proven in the argument above (with $i=M$, $j=\mu$ and $G_{\mu}(j)=0$). We state it separately for the convenience of later application. It provides an upper bound estimate for the area of a curly triangle.

\begin{lemma} \label{s2-5}
Let $0\leq M<\mu$. Then
    \begin{equation*}
        \int_{M}^{\mu} \! G_{\mu}(t) \,\textrm{d}t\leq \frac{\mu-M}{4}\frac{4-3\Theta}{2-\Theta}G_{\mu}(M)
    \end{equation*}
with $\Theta=\Theta(M, \mu, \mu)$.
\end{lemma}

Set
\begin{equation*}
    M_0=\left\lfloor G_{\mu}^{-1}(1/4)\right\rfloor+1, \quad \textrm{for $\mu\geq 2$}.
\end{equation*}
Following \cite[Lemma 4.8]{FLPS:2023}, we have $M_0<\mu$.

\begin{corollary}\label{s2-4}
\begin{equation*}
     \int_{M_0}^{\mu} \! G_{\mu}(t) \,\textrm{d}t\leq \frac{8\pi-3\sqrt{2}}{16(4\pi-\sqrt{2})}\left(\mu-M_0\right).
\end{equation*}
\end{corollary}

\begin{remark}
    Notice that the constant factor on the right-hand side is roughly  $0.117$. This improves upon the factor of $1/8$ in the proof of \cite[Theorem 6.1]{FLPS:2023}.
\end{remark}

\begin{proof}[Proof of Corollary \ref{s2-4}]
By using the definition of $g$, we have
\begin{equation*}
    \Theta(M_0, \mu, \mu)=\frac{1}{2}\left(1-\frac{M_0}{\mu} \right)\frac{(1-(M_0/\mu)^2)^{-1/2}}{\arccos(M_0/\mu)}.
\end{equation*}
Notice that
\begin{equation*}
    \arccos x\leq \frac{\pi}{2} (1-x)^{1/2}, \quad x\in [0,1].
\end{equation*}
Hence
\begin{equation*}
    \Theta(M_0, \mu, \mu)\geq \frac{1}{\pi}\left(1+\frac{M_0}{\mu} \right)^{-1/2}\geq \frac{1}{\sqrt{2}\pi}.
\end{equation*}
With this lower bound, the result then follows directly from Lemma \ref{s2-5}.
\end{proof}

The argument of \cite[Theorem 6.1]{FLPS:2023}, combined with our results above, allows us to strengthen their theorem as follows.

\begin{theorem} \label{s2-7}
Let $\mu\geq 2$. Then
\begin{equation*}  
\sum_{m=0}^{\lfloor \mu\rfloor} \left\lfloor G_{\mu}(m)+\frac{3}{4} \right\rfloor \geq \frac{\mu^2}{8}+\frac{1}{4}\left(\!\frac{24\pi-7\sqrt{2}}{4(4\pi-\sqrt{2})}+\frac{1}{2\pi\mu-1}\!\right)M_0-\frac{8\pi-3\sqrt{2}}{16(4\pi-\sqrt{2})}\mu.
\end{equation*}
\end{theorem}

\begin{proof}
Set
\begin{equation*}
N=\left\lfloor \frac{\mu}{\pi}+\frac{3}{4} \right\rfloor,
\end{equation*}
\begin{equation*}
M_k=\left\lfloor G_{\mu}^{-1}\left(k+\frac{1}{4}\right) \right\rfloor+1,\quad 1\leq k\leq N-1,
\end{equation*}
and $M_N=0$. Therefore, by Lemma \ref{s2-6},
\begin{align*}
 &\frac{1}{2} \left\lfloor G_{\mu}(0)+\frac{3}{4} \right\rfloor+   \sum_{m=1}^{\lfloor \mu\rfloor} \left\lfloor G_{\mu}(m)+\frac{3}{4} \right\rfloor\\
 =&\sum_{n=0}^{N-1}\mathbf{T}\left(G_{\mu}+\frac{3}{4},M_{n+1},M_{n}\right)\\
\geq& \sum_{n=0}^{N-1} \left( \int_{M_{n+1}}^{M_n} \! G_{\mu}(t) \,\textrm{d}t+\frac{M_n-M_{n+1}}{2(2-\beta_n)}-\frac{1}{2} \right)\\
\geq &\int_0^{M_0} \! G_{\mu}(t) \,\textrm{d}t+\frac{1}{4}\left(1+\frac{1}{2\pi\mu-1} \right)M_0 -\frac{N}{2},
\end{align*}  
where we have used
\begin{equation*}
\beta_n:=\Theta(M_{n+1}, M_n, \mu)=\frac{M_n-M_{n+1}}{2\mu}\frac{g''(M_{n+1}/\mu)}{|g'(M_{n+1}/\mu)|}\geq \frac{1}{2\mu}\frac{g''(0)}{|g'(0)|}=\frac{1}{\pi \mu}.
\end{equation*}
Combining Corollary \ref{s2-4} and evaluating the integral of $G_{\mu}$ over the interval $[0, \mu]$ then yields the desired result.
\end{proof}

Recall that \cite[Theorem 2.3]{FLPS:2023} gives that, for the unit disk $\mathcal{B}_1^2$ and all $\mu\geq 0$,
\begin{equation*}
\mathscr{N}^{\mathtt{N}}_{\mathcal{B}_1^2}(\mu)\geq \mathscr{P}^{\mathtt{N}}_2(\mu),
\end{equation*}
where the lattice point counting function $\mathscr{P}^{\mathtt{N}}_2(\mu)$ is given by
\begin{equation*}
    \mathscr{P}^{\mathtt{N}}_2(\mu)=-\left\lfloor \frac{\mu}{\pi}+\frac{3}{4} \right\rfloor+ 2\sum_{m=0}^{\lfloor \mu\rfloor} \left\lfloor G_{\mu}(m)+\frac{3}{4} \right\rfloor.
\end{equation*}
Therefore, to establish the Neumann P\'{o}lya's conjecture for the unit disk when $\mu \geq 11.89$, it is sufficient to prove that $\mathscr{P}^{\mathtt{N}}_2(\mu) \geq \mu^2/4$ for all $\mu$ in this range. By Theorem \ref{s2-7},
\begin{equation*}
\mathscr{P}^{\mathtt{N}}_2(\mu)\geq \mu^2/4 + \mathcal{R}(\mu)
\end{equation*}
with
\begin{equation*}
\mathcal{R}(\mu):=-\left\lfloor \frac{\mu}{\pi}+\frac{3}{4} \right\rfloor+\frac{1}{2}\left(\!\frac{24\pi-7\sqrt{2}}{4(4\pi-\sqrt{2})}+\frac{1}{2\pi\mu-1}\!\right)M_0-\frac{8\pi-3\sqrt{2}}{8(4\pi-\sqrt{2})}\mu.
\end{equation*}
Note that, by \cite[Lemma 4.8]{FLPS:2023},
\begin{equation}\label{s2-10}
  \mathcal{R}(\mu)>\frac{1}{2}\left(\!\frac{24\pi-7\sqrt{2}}{4(4\pi-\sqrt{2})}+\frac{1}{2\pi\mu}\!\right)\mu\cos\sigma-\frac{8\pi-3\sqrt{2}}{8(4\pi-\sqrt{2})}\mu-\left( \frac{\mu}{\pi}+\frac{3}{4} \right),
\end{equation}
whenever
\begin{equation}\label{s2-8}
    \mu\geq \frac{\pi}{4(\sin\sigma-\sigma\cos\sigma)}.
\end{equation}
Choosing $\sigma=0.59$, a calculation shows that for $\mu\geq 11.89$, both \eqref{s2-8} and $\mathcal{R}(\mu)>0$ are satisfied. This finishes the proof of Theorem \ref{thm999}. \qed


\subsection{Improvement of the Neumann P\'{o}lya's conjecture for disks}

Note that \eqref{s2-10} gives that
\begin{equation*}
    \mathcal{R}(\mu)>\left(\!\frac{24\pi-7\sqrt{2}}{8(4\pi-\sqrt{2})}\cos\sigma-\frac{8\pi-3\sqrt{2}}{8(4\pi-\sqrt{2})}-\frac{1}{\pi}\!\right)\mu-\frac{3}{4}+\frac{\cos\sigma}{4\pi}.
\end{equation*}
If we take $\sigma=0.588$, then, for $\mu\geq 12$, a calculation shows that \eqref{s2-8} holds and
\begin{equation*}
   \mathcal{R}(\mu)>0.0014\mu.
\end{equation*}
This proves Theorem \ref{thm333}. \qed


\subsection{Improvement of the Dirichlet P\'{o}lya's conjecture for disks}

We prove Theorem \ref{thm777} in this subsection. We first establish a variant of \cite[Theorem 3.4]{FLPS:2025}.

\begin{proposition} \label{s4-24}
Let $a,b\in\mathbb{Z}$ with $a<b$, and suppose that $\mathcal{G}$ is a decreasing and convex function on $[a,b]$ with Lipschitz constant $1/2$, and satisfies $\mathcal{G}(b)=0$. Suppose further that there exist $\gamma_1$ and $\gamma_2$ in $[a,b]$ with $\gamma_2<\gamma_1$ such that
\begin{equation}\label{s2-11}
    \left|\mathcal{G}(t_1)-\mathcal{G}(t_2)\right|<\frac{1}{3}\left|t_1-t_2\right| \textrm{ for all $t_1\neq t_2\in [\gamma_1, b]$},
\end{equation}
and
\begin{equation}\label{s2-12}
    \left|\mathcal{G}(t_1)-\mathcal{G}(t_2)\right|<\frac{3}{8}\left|t_1-t_2\right| \textrm{ for all $t_1\neq t_2\in [\gamma_2, b]$}.
\end{equation}
Then
\begin{equation}\label{s4-26}
\mathbf{T}\left(\mathcal{G}+\frac{1}{4},a,b\right)\le \int_a^b \! \mathcal{G}(t)\,\textrm{d} t - \frac{1}{16}\lfloor \mathcal{G}(\gamma_1)\rfloor- \frac{3}{16}\lfloor \mathcal{G}(\gamma_2)\rfloor.
\end{equation}
\end{proposition}

\begin{proof}
Following the proof strategy of \cite[Theorem 3.4]{FLPS:2025}, we first state a variant result of \cite[Lemma 3.5]{FLPS:2025}. Let $A,B\in\mathbb{Z}$ with $a\leq A<B\leq b-1$. Suppose, instead of \eqref{s2-11} and \eqref{s2-12},  the function $\mathcal{G}$ satisfies that, for some constant $c\in [1/3, 3/8]$,
\begin{equation*}
    \left|\mathcal{G}(t_1)-\mathcal{G}(t_2)\right|<c\left|t_1-t_2\right| \textrm{ for all $t_1\neq t_2\in [A+1, B+1]$}.
\end{equation*}
Suppose further that for some $n\in\mathbb Z$,
\begin{equation*}
\mathcal{G}(A)\ge n+1\ge \mathcal{G}(A+1)\ge \mathcal{G}(B)\ge n\ge \mathcal{G}(B+1).
\end{equation*}
Then
\begin{equation} \label{s4-25}
\mathbf{T}\left(\mathcal{G}+\frac{1}{4}, A, B\right)\le\int_A^B \mathcal{G}(t)\, \textrm{d} t - \left(\frac{3}{4}-\frac{3}{2}c\right).
\end{equation}
The proof of this result is essentially  the same as that of \cite[Lemma 3.5]{FLPS:2025}, requiring only minor modifications in Case 2.

Let $N=\lfloor \mathcal{G}(a)\rfloor$, and define
\begin{equation*}
 q_n=\max\{m\in\mathbb Z : \mathcal{G}(m)\ge n\},\quad \textrm{for $0\leq n\leq N$.}
\end{equation*}
Then  we have the decomposition
\begin{equation*}
\mathbf{T}\left(\mathcal{G}+\frac{1}{4},a,b\right)=\mathbf{T}\left(\mathcal{G}+\frac{1}{4},a,q_N\right) + \sum_{n=0}^{N-1}\mathbf{T}\left(\mathcal{G}+\frac{1}{4},q_{n+1},q_n\right).
\end{equation*}
Note that there are at least $\lfloor \mathcal{G}(\gamma_2)\rfloor$ values of $n$ for which $q_{n+1}+1\ge \gamma_2$, and at least $\lfloor \mathcal{G}(\gamma_1)\rfloor$ values of $n$ for which $q_{n+1}+1\ge \gamma_1$. We now apply the bound \eqref{s4-25} with $A=q_{n+1}$ and $B=q_n$ as follows: with $c=1/3$ for $\lfloor \mathcal{G}(\gamma_1)\rfloor$ values of $n$ satisfying $q_{n+1}+1\ge \gamma_1$; with $c=3/8$ for $\lfloor \mathcal{G}(\gamma_2)\rfloor-\lfloor \mathcal{G}(\gamma_1)\rfloor$ values of $n$ satisfying $q_{n+1}+1\ge \gamma_2$. For the remaining terms, we apply \cite[Theorem 3.3]{FLPS:2025}. This yields
\begin{equation*}
\mathbf{T}\left(\mathcal{G}+\frac{1}{4},a,b\right)\le \int_a^b \! \mathcal{G}(t)\,\textrm{d} t - \frac{1}{4}\lfloor \mathcal{G}(\gamma_1)\rfloor- \frac{3}{16}\left(\lfloor \mathcal{G}(\gamma_2)\rfloor-\lfloor \mathcal{G}(\gamma_1)\rfloor \right),
\end{equation*}
which  directly implies \eqref{s4-26}.
\end{proof}

We are now ready to prove Theorem \ref{thm777}. By \cite[Theorem 2.3]{FLPS:2023},
\begin{equation*}
    \mathscr{N}^{\mathtt{D}}_{\mathcal{B}_1^2}(\mu)\leq 2\mathbf{T}\left(G_\mu+\frac{1}{4},0,\lfloor \mu\rfloor+1\right).
\end{equation*}
Applying Proposition \ref{s4-24} with  $\mathcal{G}=G_\mu$, $a=0$, $b=\lfloor \mu\rfloor+1$, $\gamma_1=\mu\cos(\frac{\pi}{3})$ and $\gamma_2=\mu \cos(\frac{3\pi}{8})$ yields the desired bound. \qed


\subsection{Improvement of the Dirichlet P\'{o}lya's conjecture for balls}

We prove Theorem \ref{thm888} in this subsection. Its proof combines the argument from \cite[Section 7]{FLPS:2023} with the results from the previous subsection and \cite[Section 3]{FLPS:2025}.

We know that, by \cite[Theorems 2.3 and 7.1]{FLPS:2023}, for all $\mu\geq 0$ and $d\geq 3$,
\begin{equation}\label{s2-4-1}
\mathscr{N}^{\mathtt{D}}_{\mathcal{B}_1^d}(\mu)\leq\mathscr{P}^{\mathtt{D}}_d(\mu)=\sum_{n=0}^{\left\lfloor\mu-\frac{d}{2}+1 \right\rfloor}\binom{n+d-3}{d-3} \widetilde{\mathscr{P}}^{\mathtt{D}}_{n+\frac{d}{2}-1}(\mu),
\end{equation}
where, for $r \in[0, \mu]$,
\begin{equation*}
  \widetilde{\mathscr{P}}_r^{\mathtt{D}}(\mu):=2 \mathbf{T}\left(\widetilde{G}_{\mu, r}+\frac{1}{4},0,\lfloor\mu-r\rfloor+1\right)
\end{equation*}
and
\begin{equation*}
\widetilde{G}_{\mu, r}(t):=\left\{ \begin{array}{ll}
G_\mu(t+r),  & t \in [0, \mu-r],\\
0,        & t \in [\mu-r, \lfloor\mu-r\rfloor+1].
\end{array}\right.
\end{equation*}
We decompose the sum in \eqref{s2-4-1} as follows.
\begin{equation}\label{s2-4-2}
\sum_{n=0}^{\left\lfloor\mu-\frac{d}{2}+1 \right\rfloor}=\sum_{n=0}^{\left\lfloor\cos(\frac{3}{8}\pi)\mu-\frac{d}{2}+1 \right\rfloor}+\sum_{n=\left\lfloor\cos(\frac{3}{8}\pi)\mu-\frac{d}{2}+1 \right\rfloor+1}^{\left\lfloor\frac{\mu}{2}-\frac{d}{2}+1 \right\rfloor}+\sum_{n=\left\lfloor\frac{\mu}{2}-\frac{d}{2}+1 \right\rfloor+1}^{\left\lfloor\mu-\frac{d}{2}+1 \right\rfloor}.
\end{equation}
If each of the three summations on the right-hand side contains at least one term, we apply Proposition \ref{s4-24} to the first, \cite[Theorem 3.4]{FLPS:2025} to the second, and \cite[Theorem 3.3]{FLPS:2025} to the third. Collecting the resulting upper bounds, we obtain the following estimate
\begin{align*}
\mathscr{N}^{\mathtt{D}}_{\mathcal{B}_1^d}(\mu)\leq &2\sum_{n=0}^{\left\lfloor\mu-\frac{d}{2}+1 \right\rfloor}\binom{n+d-3}{d-3} \int_{n+\frac{d}{2}-1}^{\mu} \!\! G_{\mu}(t)\,\textrm{d} t\\
  &-\left(\frac{1}{8}\lfloor\omega_0 \mu \rfloor+\frac{3}{8}\lfloor\omega_1 \mu \rfloor \right)\sum_{n=0}^{\left\lfloor\cos(\frac{3}{8}\pi)\mu-\frac{d}{2}+1 \right\rfloor}\binom{n+d-3}{d-3}\\
  &-\frac{1}{2}\lfloor\omega_0 \mu \rfloor\sum_{n=\left\lfloor\cos(\frac{3}{8}\pi)\mu-\frac{d}{2}+1 \right\rfloor+1}^{\left\lfloor\frac{\mu}{2}-\frac{d}{2}+1 \right\rfloor}\binom{n+d-3}{d-3}.
\end{align*}
It was shown in \cite[Section 7]{FLPS:2023} that the first sum above is bounded by the desired main term of Weyl's law. Rearranging  and simplifying the binomial coefficients then yields the desired bound of the theorem.

If the first sum on the right-hand side of \eqref{s2-4-2} is empty (i.e., when $\cos(\frac{3}{8}\pi)\mu - \frac{d}{2} + 1 < 0$), we decompose \eqref{s2-4-2} into two parts according to whether $n \leq \lfloor\frac{\mu}{2} - \frac{d}{2} + 1 \rfloor$ and then apply \cite[Theorems 3.3 and 3.4]{FLPS:2025}. If the first sum is non-empty but the second is empty (i.e., when $\lfloor\cos(\frac{3}{8}\pi)\mu - \frac{d}{2} + 1 \rfloor = \lfloor\frac{\mu}{2} - \frac{d}{2} + 1 \rfloor \geq 0$), we instead apply Proposition \ref{s4-24} and \cite[Theorem 3.3]{FLPS:2025}. In both cases, the desired bound of the theorem is obtained.\qed


\section{Cylinders}

By separation of variables, the eigenvalues of the Dirichlet Laplacian in cylinders $\mathcal{C}_{R,L}^d$, $d\geq 3$,  are given by
\begin{equation*}
    x_{n,k,l}=\left(\frac{j_{n+\frac{d-3}{2},k}}{R}\right)^2+\left(\frac{\pi l}{L}\right)^2, \quad k,l\in\mathbb{N}, n\in \mathbb{N}_0:=\mathbb{N}\cup \{0\},
\end{equation*}
of multiplicity $m_n^{d-1}$, where
\begin{align*}
m_n^d:
=\left\{\begin{array}{cc}
              1 & \text{if $n=0$}, \\
              \binom{n+d-1}{d-1}-\binom{n+d-3}{d-1} &\text{if $n\ge 1$},
            \end{array}
\right.
\end{align*}
and $j_{\nu,k}$ denotes the $k$th positive zero of the Bessel function $J_{\nu}$. Notice that the squares $(j_{n+\frac{d-3}{2},k}/R)^2$ in the expression of $x_{n,k,l}$ are exactly eigenvalues of the Dirichlet Laplacian in the ball $\mathcal{B}_R^{d-1}$ with the same multiplicities  $m_n^{d-1}$.

Similarly, the eigenvalues of the Neumann Laplacian in cylinders $\mathcal{C}_{R,L}^d$ are given by
\begin{equation*}
    x'_{n,k,l}=\left(\frac{a'_{n+\frac{d-3}{2},k}}{R}\right)^2+\left(\frac{\pi l}{L}\right)^2, \quad k\in\mathbb{N}, n,l\in\mathbb{N}_0,
\end{equation*}
of multiplicity $m_n^{d-1}$, where $a'_{n+\frac{d-3}{2},k}$ denotes the $k$th positive (non-negative if $n=0$) zero of the derivative $(x^{\frac{3-d}{2}}J_{n+\frac{d-3}{2}}(x))'$. Notice that the squares $(a'_{n+\frac{d-3}{2},k}/R)^2$ in the expression of $x'_{n,k,l}$ are exactly eigenvalues of the Neumann Laplacian in the ball $\mathcal{B}_R^{d-1}$ with the same multiplicities $m_n^{d-1}$.

\subsection{Weyl's law} \label{s4.1}

The eigenvalue counting problem for a cylinder boils down to summing the counting functions for a lower-dimensional ball with varying spectral parameters.

\begin{proof}[Proof of Theorem \ref{thm111}]
In the Dirichlet case we first observe that
\begin{equation}
    \mathscr{N}_{\mathcal{C}_{R,L}^d}(\mu)=\sum_{1\leq l<L\mu/\pi} \mathscr{N}_{\mathcal{B}_R^{d-1}}\left(\sqrt{\mu^2-\left(\frac{\pi l}{L}\right)^2}\right). \label{s4-1}
\end{equation}
 By the assumption, we have
\begin{equation}
	\mathscr{N}_{\mathcal{B}_R^{d-1}}(\mu)=C_{d-1} \left|\mathcal{B}_R^{d-1}\right|\mu^{d-1}-C'_{d-1}\left|\partial\mathcal{B}_R^{d-1}\right| \mu^{d-2}+O_{\varepsilon}\left(\mu^{d-3+\alpha+\varepsilon} \right) \label{s4-5}
\end{equation}
with $C_{d-1}=\omega_{d-1}/(2\pi)^{d-1}$ and $C'_{d-1}=\omega_{d-2}/(4(2\pi)^{d-2})$. Applying this asymptotics to the summand, the sum in \eqref{s4-1} is split into three parts. The error terms contribute at most $O_{\varepsilon}(\mu^{d-2+\alpha+\varepsilon})$. It remains to estimate
\begin{equation}
 \sum_{1\leq l<L\mu/\pi} \left(\mu^2-\left(\frac{\pi l}{L}\right)^2\right)^{m/2}, \quad m\in\mathbb{N}. \label{s4-2}
\end{equation}
When $m\geq 2$, by the Euler--Maclaurin formula, it is easy to get
\begin{equation}
    \eqref{s4-2}=\frac{\Gamma(\frac{m+2}{2})}{\Gamma(\frac{m+3}{2})}\frac{L}{2\sqrt{\pi}}\mu^{m+1}-\frac{1}{2}\mu^m+O\left(\mu^{m-1}\right). \label{s4-3}
\end{equation}
When $m=1$, we obtain
\begin{equation}
\eqref{s4-2}=\frac{L}{4}\mu^2+O\left(\mu\right) \label{s4-4}
\end{equation}
by simply comparing the sum with an integral $\int_0^{L\mu/\pi}(\mu^2-(\pi x/L)^2)^{1/2}\,\textrm{d}x$. To derive the asymptotics of $\mathscr{N}_{\mathcal{C}_{R,L}^d}(\mu)$ for $d\geq 4$, we apply \eqref{s4-3} with $m=d-1$ and $m=d-2$.  For the case $d=3$ we use both \eqref{s4-3} and \eqref{s4-4}. The computation is quite straightforward if we notice the identity $\omega_{d}=\pi^{d/2}/\Gamma(\frac{d}{2}+1)$.
%

The derivation for the Neumann case follows almost the same procedure. We need only observe that in this case,
\begin{equation*}
    \mathscr{N}_{\mathcal{C}_{R,L}^d}(\mu)=\sum_{0\leq l\leq L\mu/\pi} \mathscr{N}_{\mathcal{B}_R^{d-1}}\left(\sqrt{\mu^2-\left(\frac{\pi l}{L}\right)^2}\right),
\end{equation*}
where the summand corresponding to $l=0$ leads to the plus sign of the second main term in the desired Weyl’s law, while the summand for $l=L\mu/\pi$ (if this value happens to be an integer) provides a negligible contribution. The remaining computations proceed identically.
\end{proof}


\subsection{Improvement of P\'{o}lya's conjecture} \label{s4.2}

In this subsection, we prove improved forms of P\'{o}lya's conjecture for certain product domains. As a direct consequence of the results presented here and in \cite{FLPS:2023}, Theorem \ref{thm444} follows immediately.

The following lemma is a calculus result.
\begin{lemma}\label{lemforf}
For $\sigma>0$, consider the function
\begin{equation}
   f_{\sigma}(x)=\left(1-\left(x/\sigma \right)^2\right)^{\frac{d-1}{2}}, \quad x\in \left[0, \sigma\right].  \label{s4-6}
\end{equation}
For $d\geq 3$, it is decreasing on $[0, \sigma]$, with a transition from concave down on $[0, \sigma/\sqrt{d-2}]$ to concave up on $[\sigma/\sqrt{d-2}, \sigma]$.
\end{lemma}


\begin{theorem} \label{thm555}
Let $\mathcal{C}_d=\mathscr{D}\times I\subset \mathbb{R}^{d-1}\times \mathbb{R}$ be a product domain, where $d\geq 3$ and $I$ is an interval of length $L$. If the Dirichlet P\'olya's conjecture for the domain $\mathscr{D}$ holds, then
\begin{equation}
    \mathscr{N}^{\mathtt{D}}_{\mathcal{C}_3}(\mu)\leq \frac{\omega_3 }{(2\pi)^3} |\mathcal{C}_3|\mu^3-\frac{\pi|\mathcal{C}_3|}{8L^3} \left\lfloor\frac{L\mu}{\pi} \right\rfloor\left(\left\lfloor\frac{L\mu}{\pi} \right\rfloor+\frac{1}{3} \right), \label{s4-12}
\end{equation}
and for $d\geq 4$ we have
\begin{equation}
\begin{split}
    \mathscr{N}^{\mathtt{D}}_{\mathcal{C}_d}(\mu)\leq&\frac{\omega_{d}}{(2\pi)^{d}}|\mathcal{C}_d| \mu^d-\frac{(d-1)\omega_{d-1}\left|\mathcal{C}_d\right|}{4(2\pi)^{d-1}L} \cdot\\
    &\left(\frac{\pi}{L\mu}\left\lfloor\frac{L\mu}{\pi \sqrt{d-2}} \right\rfloor\right)^2 \left(1- A_1\left(\frac{\pi}{L\mu}\left\lfloor\frac{L\mu}{\pi \sqrt{d-2}} \right\rfloor\right)^2\right)\mu^{d-1},
\end{split}\label{s4-14}
\end{equation}
where
\begin{equation}
    A_1=\left\{
            \begin{array}{ll}
           \frac{d-3}{4},  & \textrm{if $d\geq 5$,}\\
           \frac{2}{3},            &\textrm{if $d=4$.}
            \end{array}
        \right.\label{s4-17}
\end{equation}
Furthermore, if  $\mu\geq \pi \sqrt{d-2}/L$ and $d\geq 11$, the inequality \eqref{s4-14} still holds when the term below is subtracted from its right side
\begin{equation}
    \frac{\omega_{d-1}\left|\mathcal{C}_d\right|}{2(2\pi)^{d-1}L}\left(1-\left(\frac{\pi}{L\mu}\left(\left\lfloor\frac{L\mu}{\pi \sqrt{d-2}} \right\rfloor+2\right) \right)^2\right)^{\frac{d-1}{2}} \mu^{d-1}. \label{s4-16}
\end{equation}
\end{theorem}

\begin{remark}
Theorems \ref{thm555} and \ref{thm666} rely on the assumption that P\'olya's conjecture holds for the domain $\mathscr{D}$. However, under a stronger assumption that an improved version of P\'olya's conjecture holds, more refined estimates can be obtained. For example, if the domain $\mathscr{D}$ is the unit disk, one may use Filonov, Levitin, Polterovich and Sher's \cite[Theorem 1.5]{FLPS:2025}  (see \eqref{s1-4}) or Theorem \ref{thm777} to improve the bound \eqref{s4-12}.
\end{remark}

\begin{proof}[Proof of Theorem \ref{thm555}]
We may assume $I=\left[0,L\right]$ without loss of generality. Observe that the Dirichlet eigenvalues of $-\Delta^{\mathtt{D}}_{\mathcal{C}_d}$ in $\mathcal{C}_d$ are equal to
\begin{equation*}
    \lambda_{m,l}=\rho_m+\left(\frac{\pi l}{L}\right)^2, \quad m,l \in \mathbb{N},
\end{equation*}
where $\rho_m$'s  are the Dirichlet eigenvalues of $-\Delta^{\mathtt{D}}_{\mathscr{D}}$ in $\mathscr{D}$. If $L\mu/\pi<1$, then $\mathscr{N}_{\mathcal{C}_d}^{\mathtt{D}}(\mu)$ vanishes and the desired results trivially hold. Hence, we may assume $L\mu/\pi\geq 1$. Then
\begin{equation}
\mathscr{N}_{\mathcal{C}_d}^{\mathtt{D}}(\mu)=\sum_{1\leq l\leq L\mu/\pi} \mathscr{N}_{\mathscr{D}}^{\mathtt{D}}\left(\sqrt{\mu^2-\left(\frac{\pi l}{L}\right)^2}\right),\label{s4-22}
\end{equation}
where, by the assumption, we have  for all $\rho\geq0$ that
\begin{equation}
    \mathscr{N}_{\mathscr{D}}^{\mathtt{D}}(\rho) \leq \frac{\omega_{d-1}}{(2\pi)^{d-1}} \left|\mathscr{D}\right|\rho^{d-1}.\label{s4-23}
\end{equation}  

Denote
\begin{equation*}
    \Delta=\frac{L\mu}{\pi}-\left\lfloor \frac{L\mu}{\pi}\right\rfloor\in [0,1).
\end{equation*}
If $d=3$, then, by direct evaluation of the summation of squares, we obtain that
\begin{align*}
 &\sum_{1\leq l\leq \lfloor L\mu/\pi\rfloor} \mu^2-\left(\pi l/L\right)^2\\
=&\frac{2L}{3\pi}\mu^3-\frac{\pi^2}{2L^2}\left(\left\lfloor \frac{L\mu}{\pi}\right\rfloor\left(\left\lfloor \frac{L\mu}{\pi}\right\rfloor+\frac{1}{3} \right)+\frac{4}{3}\Delta^3+2\Delta^2 \left\lfloor \frac{L\mu}{\pi}\right\rfloor\right)\\
\leq &\frac{2L}{3\pi}\mu^3-\frac{\pi^2}{2L^2}\left\lfloor \frac{L\mu}{\pi}\right\rfloor\left(\left\lfloor \frac{L\mu}{\pi}\right\rfloor+\frac{1}{3} \right).
\end{align*}
Combining this bound with \eqref{s4-22} and \eqref{s4-23} yields the bound \eqref{s4-12}.


We next prove the case $d\geq 4$. For simplicity of notation, we denote
\begin{equation*}
    \sigma'=\frac{L\mu}{\pi \sqrt{d-2}} \textrm{ and } \sigma=\frac{L\mu}{\pi}.
\end{equation*}
By \eqref{s4-22} and \eqref{s4-23},
\begin{equation}
	\mathscr{N}_{\mathcal{C}_d}^{\mathtt{D}}(\mu)\leq\frac{\omega_{d-1}}{(2\pi)^{d-1}} \left|\mathscr{D}\right|\mu^{d-1}\sum_{1\leq l\leq \lfloor \sigma\rfloor} f(l)  \label{s4-8}
\end{equation}
with $f=f_{\sigma}$ defined by \eqref{s4-6}. We next estimate the sum  $\sum_l f(l)$. Roughly speaking, the sum $\sum_l f(l)$, interpreted as the combined area of rectangles, is bounded above by the integral $\int_{0}^{\sigma}\! f(x)\,\textrm{d}x$ minus the combined area of certain ``in-between'' triangles, each of width $1$. The integral produces the main term, since
\begin{equation}
\frac{\omega_{d-1}|\mathscr{D}|}{(2\pi)^{d-1}} \mu^{d-1}\!\int_{0}^{\sigma}\!\!\! f(x)\textrm{d}x=\frac{\omega_{d-1}|\mathscr{D}|}{(2\pi)^{d-1}} \cdot\frac{L\mu^d\Gamma(\frac{d+1}{2})}{2\sqrt{\pi}\Gamma(\frac{d+2}{2})}=\frac{\omega_{d}}{(2\pi)^{d}}|\mathcal{C}_d| \mu^d.  \label{s4-15}
\end{equation}
The ``in-between'' triangles  we consider belong to two classes: the first one consists of triangles under the concave-down part of $f$ on $[0, \sigma']$,  with vertices at $(l, f(l))$, $(l-1, f(l))$, and $(l-1, f(l-1))$ for $1\leq l\leq \lfloor \sigma' \rfloor$; the second one consists of triangles under the concave-up part of $f$ on $[\sigma', \sigma]$, with vertices at $(l, f(l))$, $(l-1, f(l))$, and $(l-1, f(l)+|f'(l)|)$ for $\lfloor \sigma'\rfloor+2 \leq l\leq \lfloor \sigma\rfloor$. Notice that these two classes do not always exist---their existence depends on the size of $\mu$. We need to discuss this in several cases.

When $\sigma'<1$, the second term of the right side of \eqref{s4-14} vanishes, hence \eqref{s4-14} follows easily from the observation that $\sum_l f(l)\leq \int_{0}^{\sigma}\! f(x)\,\textrm{d}x$.

When $\sigma'\geq 1$, the first class is non-empty. The combined area of the triangles in the first class is equal to
\begin{equation}
    \frac{1}{2}\sum_{1\leq l \leq \lfloor \sigma' \rfloor}f(l-1)-f(l)=\frac{1}{2}\left(1-\left(1-\left(\lfloor \sigma' \rfloor/\sigma\right)^2\right)^{\frac{d-1}{2}} \right). \label{s4-13}
\end{equation}
If $d\geq 5$, by using the Taylor expansion retaining three terms plus a remainder term, we drop the remainder and obtain
\begin{equation*}
    \eqref{s4-13}\geq \frac{d-1}{4}\left(\frac{\lfloor \sigma' \rfloor}{\sigma} \right)^2 \left(1-\frac{d-3}{4}\left(\frac{\lfloor \sigma' \rfloor}{\sigma} \right)^2 \right).
\end{equation*}
If $d=4$, by using the difference of cubes formula and the Taylor expansion retaining two terms plus a remainder term, we get
\begin{align*}
    \eqref{s4-13}&=\frac{1}{2}\! \left(\!\!1-\left(1-\left(\frac{\lfloor \sigma' \rfloor}{\sigma} \right)^{\!2} \right)^{\frac{1}{2}}\! \right)\left(\! 2+\left(1-\left(\frac{\lfloor \sigma' \rfloor}{\sigma} \right)^{\!2} \right)^{\frac{1}{2}}-\left(\frac{\lfloor \sigma' \rfloor}{\sigma} \right)^{\!2}\!\right)\\
    &\geq \frac{3}{4}\left(\frac{\lfloor \sigma' \rfloor}{\sigma} \right)^2\left( 1-\frac{2}{3}\left(\frac{\lfloor \sigma' \rfloor}{\sigma} \right)^2\right).
\end{align*}
Combining \eqref{s4-15} and the above two bounds yields \eqref{s4-14}.

When $\sigma' \geq 1$ and $d \geq 11$, we have $\sigma - \sigma' \geq 2$. This implies that $\lfloor \sigma' \rfloor + 2 \leq \lfloor \sigma \rfloor$, ensuring that the second class is non-empty. The combined area of the triangles in the second class is equal to
\begin{align*}
     \frac{1}{2}\sum_{\lfloor \sigma' \rfloor + 2\leq l \leq \lfloor \sigma \rfloor}&\left|f'(l)\right|\geq \frac{1}{2}\sum_{\lfloor \sigma' \rfloor + 2\leq l \leq \lfloor \sigma \rfloor-1}\left( f(l)-f(l+1)\right)+\frac{1}{2}\left|f'(\lfloor \sigma \rfloor)\right|\\
     &\geq \frac{1}{2}\left(f\left(\lfloor \sigma' \rfloor + 2 \right) -f\left(\lfloor \sigma \rfloor\right)+\left|f'(\lfloor \sigma \rfloor)\right|\right)\geq \frac{1}{2} f\left(\lfloor \sigma' \rfloor + 2 \right)\\
     &=\frac{1}{2}\left(1-\left(\frac{\lfloor \sigma' \rfloor+2}{\sigma}\right)^2\right)^{\frac{d-1}{2}},
\end{align*}
where the convexity is utilized in the above derivation. This bound leads to \eqref{s4-16}, thereby concluding the proof of the theorem.
\end{proof}


\begin{remark}
In the case of $d=3$, our proof involves directly computing the exact value of the sum $\sum_l f(l)$. Of course, we could also handle the $d=3$ case similarly to how we treat $d\geq 4$. In fact, by comparing the sum with an integral, we can obtain the following bound slightly weaker than  \eqref{s4-12}:
\begin{equation}
    \mathscr{N}^{\mathtt{D}}_{\mathcal{C}_3}(\mu)\leq \frac{\omega_3 }{(2\pi)^3} |\mathcal{C}_3|\mu^3-\frac{\pi|\mathcal{C}_3|}{8L^3} \left\lfloor\frac{L\mu}{\pi} \right\rfloor^2. \label{s4-7}
\end{equation}
The proof is as follows. We may assume $\mu>\pi/L$, as otherwise \eqref{s4-7} holds trivially. Indeed, if $\mu\leq \pi/L$, then $\mathscr{N}^{\mathtt{D}}_{\mathcal{C}_3}(\mu)$ equals zero, while the right-hand side of \eqref{s4-7} is nonnegative. Since the curve of $f$ is concave down over the entire interval $[0, \sigma]$, the sum $\sum_l f(l)$ is bounded above by the integral $\int_{0}^{\sigma}\! f(x)\,\textrm{d}x$ minus the combined area of  triangles with vertices at $(l, f(l))$, $(l-1, f(l))$ and $(l-1, f(l-1))$ for $1\leq l\leq \lfloor \sigma\rfloor$. In other words,
\begin{align*}
\mu^2 \sum_{1\leq l\leq \lfloor \sigma \rfloor} f(l)&<\mu^2 \int_{0}^{\sigma}\! f(x)\,\textrm{d}x-\mu^2 \sum_{l=1}^{\lfloor \sigma\rfloor} \frac{1}{2}\left(f(l-1)-f(l) \right)\\
     &=\frac{2L}{3\pi}\mu^3-\frac{\pi^2}{2L^2}\left\lfloor\frac{L\mu}{\pi} \right\rfloor^2.
\end{align*}
Plugging this bound in \eqref{s4-8} gives \eqref{s4-7}. \qed
\end{remark}


\begin{theorem} \label{thm666}
Let $\mathcal{C}_d=\mathscr{D}\times I\subset \mathbb{R}^{d-1}\times \mathbb{R}$ be a product domain, where $d\geq 3$ and $I$ is an interval of length $L$.  If the Neumann P\'olya's conjecture for the domain $\mathscr{D}$ holds, then
\begin{equation}
\mathscr{N}^{\mathtt{N}}_{\mathcal{C}_3}(\mu)-\frac{\omega_3 }{(2\pi)^3} |\mathcal{C}_3|\mu^3\geq \left\{
            \begin{array}{ll}
           \frac{|\mathcal{C}_3|}{8L^2} \left( \frac{L\mu}{\pi}-\frac{1}{3}\right)\mu,                                         & \textrm{if $\frac{L\mu}{\pi} \geq \frac{1}{3}$,}\\
           \frac{|\mathcal{C}_3|}{6\pi L} \left( \frac{3}{2}-\frac{L\mu}{\pi}\right)\mu^2,                                   &\textrm{if $0\leq \frac{L\mu}{\pi} \leq \frac{3}{2}$,}
            \end{array}
        \right. \label{s4-9}
\end{equation}
and, for $d\geq 4$, we have that
\begin{equation}
\mathscr{N}^{\mathtt{N}}_{\mathcal{C}_d}(\mu)\geq \frac{\omega_{d}}{(2\pi)^{d}} \left|\mathcal{C}_d\right|\mu^d+\frac{\omega_{d}\left|\mathcal{C}_d\right|}{2(2\pi)^{d-1}L}\left(\frac{2\Gamma(\frac{d+2}{2})}{\sqrt{\pi}\Gamma(\frac{d+1}{2})}-\frac{L\mu}{\pi}\right) \mu^{d-1}  \label{s4-20}
\end{equation}
when  $0\leq \frac{L\mu}{\pi}\leq \frac{2\Gamma(\frac{d+2}{2})}{\sqrt{\pi}\Gamma(\frac{d+1}{2})}$, and that
\begin{equation}
	\begin{split}
		\mathscr{N}^{\mathtt{N}}_{\mathcal{C}_d}(\mu)\geq &\frac{\omega_{d}}{(2\pi)^{d}} \left|\mathcal{C}_d\right|\mu^d+\frac{(d-1)\omega_{d-1}\left|\mathcal{C}_d\right|}{4(2\pi)^{d-1}L} \cdot \\
		&\left(\frac{\left\lfloor\frac{L\mu}{\pi \sqrt{d-2}} \right\rfloor-1}{L\mu/\pi}\right)^{\!\!2}\left(1-A_1 \left(\frac{\left\lfloor\frac{L\mu}{\pi \sqrt{d-2}} \right\rfloor-1}{L\mu/\pi}\right)^{\!\!2}\right)\mu^{d-1}
	\end{split}\label{s4-18}
\end{equation}
when  $\mu\geq \pi \sqrt{d-2}/L$, with $A_1$ defined by \eqref{s4-17}. Furthermore, if  $\mu\geq \pi \sqrt{d-2}/L$ and $d\geq 11$, the inequality \eqref{s4-18}  still holds when the term below is added to its right side
\begin{equation}
\frac{\omega_{d-1}\left|\mathcal{C}_d\right|}{2(2\pi)^{d-1}L}\left(1- \left( \frac{\pi}{L\mu} \left(\left\lfloor\frac{L\mu}{\pi \sqrt{d-2}} \right\rfloor+1\right)\right)^{\!\!2}\right)^{\frac{d-1}{2}}\mu^{d-1}. \label{s4-21}
\end{equation}
\end{theorem}

\begin{remark}
If $d=2k+1$, $k\in \mathbb{N}$, then
\begin{equation*}
	\frac{\Gamma(\frac{d+2}{2})}{\Gamma(\frac{d+1}{2})}=\frac{\sqrt{\pi}}{2}\frac{(2k+1)!!}{(2k)!!};
\end{equation*}
if $d=2k$, $k\in \mathbb{N}$, then
\begin{equation*}
	\frac{\Gamma(\frac{d+2}{2})}{\Gamma(\frac{d+1}{2})}=\frac{1}{\sqrt{\pi}}\frac{(2k)!!}{(2k-1)!!}.
\end{equation*}
\end{remark}

\begin{proof} [Proof of Theorem  \ref{thm666}]
We may assume $I=\left[0,L\right]$ without loss of generality. Observe that the Neumann eigenvalues of $-\Delta^{\mathtt{N}}_{\mathcal{C}_d}$ in $\mathcal{C}_d$ are equal to
\begin{equation*}
    \lambda'_{m,l}=\rho'_m+\left(\frac{\pi l}{L}\right)^2, \quad m,l \in \mathbb{N}_0,
\end{equation*}
where $\rho'_m$'s  are the Neumann eigenvalues of $-\Delta^{\mathtt{N}}_{\mathscr{D}}$ in $\mathscr{D}$. Therefore,
\begin{equation}
\mathscr{N}_{\mathcal{C}_d}^{\mathtt{N}}(\mu)=\sum_{0\leq l\leq L\mu/\pi} \mathscr{N}_{\mathscr{D}}^{\mathtt{N}}\left(\sqrt{\mu^2-\left(\frac{\pi l}{L}\right)^2}\right), \label{s4-10}
\end{equation}
where, by the assumption, the function $\mathscr{N}_{\mathscr{D}}^{\mathtt{N}}(\rho)$ satisfies for all $\rho\geq 0$ that
\begin{equation}
    \mathscr{N}_{\mathscr{D}}^{\mathtt{N}}(\rho) \geq \frac{\omega_{d-1}}{(2\pi)^{d-1}} \left|\mathscr{D}\right|\rho^{d-1}. \label{s4-11}
\end{equation}

We first consider the case $d=3$. The formula \eqref{s4-10} gives particularly that
\begin{align*}
\mathscr{N}_{\mathcal{C}_3}^{\mathtt{N}}(\mu)&\geq \mathscr{N}_{\mathscr{D}}^{\mathtt{N}}(\mu)\geq \frac{\omega_2}{(2\pi)^2} \left|\mathscr{D}\right|\mu^{2}\\
  &=\frac{\omega_3 }{(2\pi)^3} |\mathcal{C}_3|\mu^3+\left(\frac{\omega_2}{(2\pi)^2} \left|\mathscr{D}\right|\mu^{2}-\frac{\omega_3 }{(2\pi)^3} |\mathcal{C}_3|\mu^3 \right).
\end{align*}
Simplifying this expression gives the second case of \eqref{s4-9},  which is a desirable lower bound only when $L\mu/\pi\leq 3/2$.

In fact, for all $\mu\geq 0$,  by \eqref{s4-10} and \eqref{s4-11}, we have
\begin{equation*}
    \mathscr{N}_{\mathcal{C}_3}^{\mathtt{N}}(\mu)\geq \frac{|\mathscr{D}|}{4\pi}\sum_{0\leq l\leq \lfloor L\mu/\pi\rfloor}  \left(\mu^2-\frac{\pi^2}{L^2} l^2 \right).
\end{equation*}
Denote
\begin{equation*}
    \Delta=\frac{L\mu}{\pi}-\left\lfloor \frac{L\mu}{\pi}\right\rfloor\in [0,1).
\end{equation*}
By direct evaluation of the summation of squares combined with factorization, we obtain that
\begin{align*}
    \mathscr{N}_{\mathcal{C}_3}^{\mathtt{N}}(\mu)&\geq \frac{|\mathscr{D}|}{4\pi}\left( \frac{2L}{3\pi}\mu^3+\frac{1}{2}\mu^2-\frac{\pi}{6L}\mu+\frac{\pi^2}{6L^2}\Delta(1-\Delta)\!\left(\!4\frac{L\mu}{\pi}+2\left\lfloor \frac{L\mu}{\pi}\right\rfloor+1 \right)\! \right)\\
    &\geq \frac{|\mathscr{D}|}{4\pi}\left( \frac{2L}{3\pi}\mu^3+\frac{1}{2}\mu^2-\frac{\pi}{6L}\mu\right)\\
    &=\frac{\omega_3 }{(2\pi)^3} |\mathcal{C}_3|\mu^3+\frac{|\mathcal{C}_3|}{8 L^2}\mu \left( \frac{L\mu}{\pi}-\frac{1}{3}\right).
\end{align*}
While this lower bound is always valid, it yields meaningful results only when $L\mu/\pi$ exceeds $1/3$. This finishes the proof of \eqref{s4-9}.


We next prove the case $d\geq 4$. For simplicity of notation, we denote
\begin{equation*}
	\sigma'=\frac{L\mu}{\pi \sqrt{d-2}} \textrm{ and } \sigma=\frac{L\mu}{\pi}.
\end{equation*}
By \eqref{s4-10} and \eqref{s4-11}, for all $\mu\geq 0$,
\begin{equation*}
	\mathscr{N}_{\mathcal{C}_d}^{\mathtt{N}}(\mu)\geq \frac{\omega_{d-1}}{(2\pi)^{d-1}} \left|\mathscr{D}\right|\mu^{d-1}\sum_{0\leq l\leq \lfloor \sigma\rfloor} f(l)
\end{equation*}
with $f=f_{\sigma}$ defined by \eqref{s4-6}. In particular,
\begin{align*}
\mathscr{N}_{\mathcal{C}_d}^{\mathtt{N}}(\mu)&\geq \frac{\omega_{d-1}}{(2\pi)^{d-1}} \left|\mathscr{D}\right|\mu^{d-1}\\
&=\frac{\omega_{d}}{(2\pi)^{d}} \left|\mathcal{C}_d\right|\mu^d+\left(\frac{\omega_{d-1}}{(2\pi)^{d-1}} \left|\mathscr{D}\right|\mu^{d-1}-\frac{\omega_{d}}{(2\pi)^{d}} \left|\mathcal{C}_d\right|\mu^d \right).
\end{align*}
Combining and simplifying the last two terms yields \eqref{s4-20}.

The idea of the following proof is similar to that of Theorem \ref{thm555}. The sum $\sum_l f(l)$ is bounded below by the integral $\int_{0}^{L\mu/\pi}\! f(x)\,\textrm{d}x$ plus the combined area of certain ``in-between'' triangles, each of width $1$. Those triangles belong to two classes: the first one consists of triangles above the concave-down part of $f$ on $[0, \sigma']$,  with vertices at $(l, f(l))$, $(l+1, f(l))$, and $(l+1, f(l)-|f'(l)|)$ for $0\leq l\leq \lfloor \sigma' \rfloor-1$; the second one consists of triangles above the concave-up part of $f$ on $[\sigma', \sigma]$, with vertices at $(l, f(l))$, $(l+1, f(l))$, and $(l+1, f(l+1))$ for $\lfloor \sigma'\rfloor+1 \leq l\leq \lfloor \sigma\rfloor-1$, along with an additional triangle with vertices at $(\lfloor \sigma\rfloor, f(\lfloor \sigma\rfloor))$, $(\lfloor \sigma\rfloor+1, f(\lfloor \sigma\rfloor))$ and $(\lfloor \sigma\rfloor+1,0)$.

When $\sigma'\geq 2$, the combined area of the triangles in the first class is equal to
\begin{align}
	\frac{1}{2}\sum_{0\leq l \leq \lfloor \sigma' \rfloor-1}|f'(l)|&\geq \frac{1}{2}\sum_{1\leq l\leq \lfloor \sigma' \rfloor-1} f(l-1)-f(l)\nonumber\\
	&= \frac{1}{2}\left(1-\left(1-\left(\frac{\lfloor \sigma' \rfloor-1}{\sigma}\right)^2\right)^{\frac{d-1}{2}} \right)\nonumber\\
	&\geq \frac{d-1}{4}\left(\frac{\lfloor \sigma' \rfloor-1}{\sigma} \right)^2 \left(1-A_1 \left(\frac{\lfloor \sigma' \rfloor-1}{\sigma} \right)^2 \right), \label{s4-19}
\end{align}
where the last inequality follows from the same method used for \eqref{s4-13}. It is trivial that the lower bound \eqref{s4-19} holds for $1\leq \sigma'<2$. Hence we obtain \eqref{s4-18}.

When $\sigma' \geq 1$ and $d \geq 11$, we have $\sigma- \sigma' \geq 2$. This implies that $\lfloor \sigma' \rfloor + 1 \leq \lfloor \sigma  \rfloor-1$. The combined area of the triangles in the second class is equal to
\begin{align*}
&\frac{1}{2}\sum_{\lfloor \sigma' \rfloor + 1\leq l \leq \lfloor \sigma \rfloor-1}\left( f(l)-f(l+1)\right)+\frac{1}{2} f(\lfloor \sigma \rfloor)=\frac{1}{2} f\left(\lfloor \sigma' \rfloor+1 \right)\\
	=&\frac{1}{2}\left(1-\left(\frac{\lfloor \sigma' \rfloor+1}{\sigma}\right)^2\right)^{\frac{d-1}{2}}.
\end{align*}
This leads to \eqref{s4-21}, thereby concluding the proof of the theorem.
\end{proof}


\subsection*{Acknowledgments}
We thank Prof. Renjin Jiang for helpful comments, which included drawing our attention to their result \cite[Theorem 2.3]{JL:2025}. C. Miao is partially supported by the National Key R\&D Program of China (no. 2022YFA1005700) and the NSFC (no. 12371095). J. Guo is partially supported by the NSFC (nos. 12571110 and 12341102).


%
%
%

\end{document}